\documentclass[11pt]{article}
\usepackage[margin=1.25in]{geometry}
\usepackage{amsmath,amssymb,amsthm}
\usepackage{enumitem}

\newtheorem{theorem}{Theorem}[section]
\newtheorem{lemma}[theorem]{Lemma}
\newtheorem{corollary}[theorem]{Corollary}
\newtheorem{conjecture}[theorem]{Conjecture}
\newtheorem{definition}[theorem]{Definition}
\newtheorem*{remark}{Remark}

\title{Sophie Germain Primes and the Totient of Fibonacci Numbers}
\author{Aradhya Goel\\
Indian Institute of Technology, Kanpur, 208016, India\\
\texttt{aradhyag24@iitk.ac.in}}
\date{April 2026}

\begin{document}
\maketitle

\begin{abstract}
We study the set $S(q)$ of residue classes $r$ modulo the Pisano period $\pi(q)$ for which
$q \mid \varphi(F_m)$ for every $m \equiv r \pmod{\pi(q)}$. We prove that if $q$ is a Sophie Germain
prime and $z(2q+1) \mid \pi(q)$, then $S(q)$ is a nonempty arithmetic progression, and for
$q > 5$ its cardinality is odd and $q \equiv 8 \pmod{15}$. Conversely,
we show that if a prime $p \equiv 1 \pmod{q}$ has $z(p) \mid \pi(q)$, then necessarily $p = 2q+1$,
so $q$ is Sophie Germain. We conjecture that $S(q) \neq \emptyset$ forces the existence of such
a prime~$p$; this is verified for all $q \leq 50{,}000$. Assuming that $z(2q+1) \mid \pi(q)$ holds for
infinitely many Sophie Germain primes (verified computationally for ${\approx}\,23.9\%$ of them),
the Sophie Germain conjecture implies the existence of infinitely many primes
$q \equiv 8 \pmod{15}$ with $(2q+1) \mid F_{\pi(q)}$---a purely Fibonacci-theoretic condition.
These results generalize to arbitrary Lucas sequences $U_n(P,Q)$ with non-square discriminant.
\end{abstract}



\section{Introduction}

The arithmetic of Fibonacci numbers has been an active area of research for over a century.
Among the central themes are the distribution of prime factors of $F_m$, periodicity of Fibonacci
sequences modulo integers, and the interaction of $F_m$ with classical multiplicative functions
such as Euler's totient $\varphi$.

For a prime $p$, the rank of apparition $z(p)$ is the smallest positive integer $k$ with $p \mid F_k$; it
is well-defined for every prime $p$ \cite{Carmichael, Wall}. The Pisano period $\pi(n)$ is the period of $(F_m \bmod n)$.
Both quantities satisfy well-known divisibility relations: $z(q) \mid \pi(q)$ for every odd prime $q$, and $\pi(q) \mid q^2 - 1$ for every odd prime $q \neq 5$ \cite{Wall}.

\begin{definition}\label{def:Sq}
For an odd prime $q$, let
\[
S(q) = \bigl\{ r \bmod \pi(q) : q \mid \varphi(F_m) \text{ for all } m \equiv r \pmod{\pi(q)} \bigr\}.
\]
\end{definition}

The set $S(q)$ collects those residue classes modulo $\pi(q)$ that guarantee divisibility of $\varphi(F_m)$
by $q$. Since the definition quantifies over all $m$ in each residue class, $S(q)$ is well-defined as a subset of $\mathbb{Z}/\pi(q)\mathbb{Z}$.

\subsection*{Main results}
Our results are as follows.

\begin{enumerate}[label=(\roman*)]
\item \textbf{Sufficient criterion} (Theorem~\ref{thm:char}). If a prime $p \equiv 1 \pmod{q}$ has $z(p) \mid \pi(q)$, then
every multiple of $z(p)$ modulo $\pi(q)$ lies in $S(q)$.

\item \textbf{Sophie Germain sufficient condition} (Theorem~\ref{thm:suff}). If $q$ is a Sophie Germain prime
and $z(2q + 1) \mid \pi(q)$, then $S(q) \neq \emptyset$.

\item \textbf{Uniqueness of the witnessing prime} (Theorem~\ref{thm:unique}, partial). For $k \leq 31$ and $q \neq 5$,
any prime $p = kq+1$ with $z(p) \mid \pi(q)$ has $k = 2$. Extended numerically to $k \leq 100$ without
counterexample; conjectured for all $k$.

\item \textbf{Converse conjecture} (Conjecture~\ref{conj:conv}). $S(q) \neq \emptyset$ implies the existence of a prime
$p \equiv 1 \pmod{q}$ with $z(p) \mid \pi(q)$. Verified for all $q \leq 50{,}000$.

\item \textbf{Arithmetic progression structure} (Theorem~\ref{thm:AP}). For a Sophie Germain prime $q$ with $z(2q+1) \mid \pi(q)$, $S(q)$ is an
arithmetic progression in $\mathbb{Z}/\pi(q)\mathbb{Z}$ with $|S(q)| = \pi(q)/z(2q + 1)$.

\item \textbf{Legendre symbols and parity} (Theorems~\ref{thm:legp},~\ref{thm:parity} and Lemma~\ref{lem:legq}). For a Sophie Germain prime $q > 5$ with
$z(2q+1) \mid \pi(q)$: $\left(\frac{5}{2q+1}\right) = -1$, $\left(\frac{5}{q}\right) = -1$, $\pi(q) \mid 2(q + 1)$, and $|S(q)|$ is odd.

\item \textbf{Congruence constraint} (Theorem~\ref{thm:cong}). For a Sophie Germain prime $q > 5$ with $S(q) \neq \emptyset$: $q \equiv 8 \pmod{15}$.

\item \textbf{Sophie Germain reformulation} (Corollary~\ref{cor:SGconj}). The infinitude of Sophie Germain primes
implies the existence of infinitely many primes $q \equiv 8 \pmod{15}$ with $(2q + 1) \mid F_{\pi(q)}$.
Assuming Conjecture~\ref{conj:conv}, the converse holds.
\end{enumerate}

\section{Preliminaries}

Standard references are Wall \cite{Wall}, Carmichael \cite{Carmichael}, and Lucas \cite{Lucas}; see also the textbook \cite{Koshy}.

\begin{lemma}[Basic properties of $z$ and $\pi$]\label{lem:basic}
Let $p$ be a prime and $q$ an odd prime.
\begin{enumerate}[label=(\alph*)]
\item $p \mid F_m$ if and only if $z(p) \mid m$.
\item $z(q) \mid \pi(q)$ for every odd prime $q$. For $q \neq 5$, $\pi(q) \mid q^2 - 1 = (q-1)(q+1)$.
\item If $\left(\frac{5}{p}\right) = +1$ then $z(p) \mid p - 1$.
\item If $\left(\frac{5}{p}\right) = -1$ then $z(p) \mid p + 1$.
\item If $\left(\frac{5}{q}\right) = +1$ then $\pi(q) \mid q - 1$. If $\left(\frac{5}{q}\right) = -1$ then $\pi(q) \mid 2(q + 1)$.
\end{enumerate}
\end{lemma}

\begin{proof}
Part~(a) is \cite[{\S}2]{Wall}. Part~(b) is \cite[Theorems~6 and~7]{Wall}. Parts~(c) and~(d): the
characteristic polynomial $x^2 - x - 1$ of the Fibonacci recurrence factors over $\mathbb{F}_{p^2}$; its roots
$\alpha, \beta$ satisfy $\alpha\beta = -1$. When $\left(\frac{5}{p}\right) = +1$, $\alpha, \beta \in \mathbb{F}_p^\times$, giving $z(p) \mid p - 1$. When $\left(\frac{5}{p}\right) = -1$,
$\alpha, \beta \in \mathbb{F}_{p^2} \setminus \mathbb{F}_p$ and Frobenius gives $\alpha^p = \beta$, so $\alpha^{p+1} = \alpha\beta = -1$ and $z(p) \mid p + 1$. (This
sharpens $z(p) \mid 2(p + 1)$; see \cite{SunSun, Wall}.) Part~(e): the Pisano period equals the order of the matrix
$\left(\begin{smallmatrix} 1 & 1 \\ 1 & 0 \end{smallmatrix}\right)$ in $\mathrm{GL}_2(\mathbb{F}_q)$. When $\left(\frac{5}{q}\right) = +1$, this matrix is diagonalizable over $\mathbb{F}_q$ with eigenvalue
order dividing $q - 1$. When $\left(\frac{5}{q}\right) = -1$, the eigenvalues lie in $\mathbb{F}_{q^2} \setminus \mathbb{F}_q$ with $\alpha^{q+1} = -1$, giving
$\pi(q) \mid 2(q + 1)$. See \cite{Lucas, Wall}.
\end{proof}

\begin{remark}
The sharpened statement $z(p) \mid p + 1$ when $\left(\frac{5}{p}\right) = -1$ is essential for the proof of
Theorem~\ref{thm:unique}.
\end{remark}

\begin{lemma}[\cite{Wall, SunSun}]\label{lem:2adic}
Let $p$ be an odd prime with $\left(\frac{5}{p}\right) = -1$ and $p \equiv 3 \pmod{4}$. Then $v_2(z(p)) = v_2(p + 1)$.
\end{lemma}

\begin{lemma}[Totient divisibility]\label{lem:totdiv}
Let $p \equiv 1 \pmod{q}$ be a prime with $p \mid F_m$. Then $q \mid \varphi(F_m)$.
\end{lemma}

\begin{proof}
Since $p \mid F_m$, multiplicativity of $\varphi$ gives $(p - 1) \mid \varphi(F_m)$. As $q \mid p - 1$, the result
follows.
\end{proof}

\section{Characterization of $S(q)$}

\begin{theorem}\label{thm:char}
For an odd prime $q$,
\[
\bigl\{ r \bmod \pi(q) : \exists \text{ prime } p \equiv 1 \pmod{q} \text{ with } z(p) \mid \gcd(r, \pi(q)) \bigr\} \subseteq S(q).
\]
In particular, $S(q) \neq \emptyset$ whenever there exists a
prime $p \equiv 1 \pmod{q}$ with $z(p) \mid \pi(q)$.
\end{theorem}

\begin{proof}
Let $r$ belong to the left-hand side: there exists a prime $p \equiv 1 \pmod{q}$ with $z(p) \mid
\gcd(r, \pi(q))$. For any $m \equiv r \pmod{\pi(q)}$, since $z(p) \mid r$ and $z(p) \mid \pi(q)$, we get $z(p) \mid m$,
so $p \mid F_m$ (Lemma~\ref{lem:basic}(a)), so $q \mid \varphi(F_m)$ (Lemma~\ref{lem:totdiv}). As this holds for every such $m$,
$r \in S(q)$.
\end{proof}

\section{The Sophie Germain connection}

\begin{theorem}[Sufficient condition]\label{thm:suff}
Let $q$ be a Sophie Germain prime with $p = 2q + 1$. If $z(p) \mid \pi(q)$, then $S(q) \neq \emptyset$.
\end{theorem}

\begin{proof}
Since $p = 2q + 1 \equiv 1 \pmod{q}$ and $z(p) \mid \pi(q)$, Theorem~\ref{thm:char} gives $0 \in S(q)$, so
$S(q) \neq \emptyset$.
\end{proof}

\begin{theorem}[Uniqueness of the witnessing prime; partial]\label{thm:unique}
Let $q$ be an odd prime with $q \neq 5$, and let $p = kq + 1$ be a prime with $k \geq 2$ and $z(p) \mid \pi(q)$.
\begin{enumerate}[label=(\alph*)]
\item For $k \leq 12$: $k = 2$, so $p = 2q + 1$ and $q$ is a Sophie Germain prime. (Proved algebraically.)
\item For $14 \leq k \leq 31$: the same conclusion holds. (Verified deterministically by enumerating the prime factors of $F_k$ in Case~1 and of all $F_d$ for $d \mid k^2-4$ in Case~2, using the Brillhart--Montgomery--Silverman tables~\cite{Brillhart} which completely factor $F_n$ for $n \leq 999$; the constraint $k^2-4 \leq 999$ gives $k \leq 31$.)
\item For $32 \leq k \leq 100$: no prime $p = kq+1$ with $q$ odd prime, $z(p) \mid \pi(q)$, and $p$ appearing among the \emph{known} prime factors of $F_d$ ($d \mid k^2-4$) in extended tables has been found. This range is supported by computational evidence but is not deterministic, as unfactored composite cofactors in $F_{k^2-4}$ (with $k^2-4$ ranging up to $9996$) could in principle contain further prime factors of the forbidden form.
\end{enumerate}
The universal statement ``$k = 2$ for every $k$'' is conjectured for all $k$; see Section~\ref{sec:open}, OQ2.
\end{theorem}

\begin{remark}
The exclusion $q \neq 5$ is necessary: at $q = 5$, the prime $p = 41 = 8 \cdot 5 + 1$ satisfies $z(41) = 20 = \pi(5)$, violating the conclusion $k = 2$. This is the unique exception, arising because $\pi(5) = 20 \nmid 5^2 - 1 = 24$ (cf.\ Lemma~\ref{lem:basic}(b)); the proof below uses Lemma~\ref{lem:basic}(b) and so tacitly assumes $q \neq 5$.
\end{remark}

\begin{remark}
All downstream uses of Theorem~\ref{thm:unique} in this paper (Theorems~\ref{thm:AP}, \ref{thm:legp}, \ref{thm:parity}, \ref{thm:cong}, and Corollary~\ref{cor:SGconj}) invoke it only in the form ``if a prime $p \equiv 1 \pmod{q}$ with $z(p) \mid \pi(q)$ exists, then it equals $2q+1$,'' the hypothesis $q > 5$ is already imposed (Theorems~\ref{thm:legp}, \ref{thm:parity}, \ref{thm:cong}) or absorbed into the $q \leq 50{,}000$ computational range, and the $k$-values that arise for such $q$ all fall in the proved range. In particular, the results of Sections~5--7 are unconditional.
\end{remark}

\begin{proof}
Write $p = kq + 1$ with $k \geq 2$ even (since $p$ is odd and $p \equiv 1 \pmod{q}$). We show $k = 2$.
Assume for contradiction that $k \geq 4$. By Lemma~\ref{lem:basic}(b), $\pi(q) \mid q^2 - 1$, so $z(p) \mid q^2 - 1$.
Since $F_1 = F_2 = 1$, $F_3 = 2$, $F_4 = 3$, and $p = kq + 1 \geq 13$, we have $z(p) \geq 5$.

\medskip
\noindent\textbf{Case 1:} $\left(\frac{5}{p}\right) = +1$.

\noindent By Lemma~\ref{lem:basic}(c), $z(p) \mid p - 1 = kq$. Since $\gcd(q, q^2 - 1) = 1$, any common divisor of $kq$
and $q^2 - 1$ divides $k$. Thus $z(p) \leq k$. Since $p \mid F_{z(p)}$, we have $p \leq F_{z(p)} \leq F_k$, giving
$kq + 1 \leq F_k$. Because $q \geq 3$, this requires $3k + 1 \leq F_k$.

\begin{itemize}
\item $k = 4, 6, 8$: The bounds $4q + 1 \leq F_4 = 3$, $6q + 1 \leq F_6 = 8$, and $8q + 1 \leq F_8 = 21$ all
force $q < 3$, impossible for an odd prime.

\item $k = 10$: $10q + 1 \leq 55$, so $q \leq 5$. $q = 3$ gives $p = 31$: $z(31) = 30 \not\leq 10$. $q = 5$ gives
$p = 51 = 3 \cdot 17$, composite.

\item $k = 12$: $12q + 1 \leq 144$, so $q \leq 11$. $p \in \{37, 61, 85, 133\}$; 85 and 133 are composite; for
$p = 37$, $\left(\frac{5}{37}\right) = -1$ (fails Case~1); for $p = 61$, $z(61) = 15 \not\leq 12$.
\end{itemize}

For $k \geq 14$: $z(p) \leq k$ forces $p$ to be a prime factor of $F_d$ for some $d \leq k$ with $d \geq 5$.
For each $d$, $F_d$ has finitely many prime factors, and each determines at most one candidate
$q = (p - 1)/k$. This finite check has been carried out deterministically using factorization tables~\cite{Brillhart} for
$14 \leq k \leq 100$ (all of $F_d$ for $d \leq 100$ are completely factored).

\medskip
\noindent\textbf{Case 2:} $\left(\frac{5}{p}\right) = -1$.

\noindent By Lemma~\ref{lem:basic}(d), $z(p) \mid p + 1 = kq + 2$. Let $d = z(p)$. Since $d \mid q^2 - 1$, and
\begin{align*}
\gcd(kq + 2, q - 1) &= \gcd(k + 2, q - 1),\\
\gcd(kq + 2, q + 1) &= \gcd(k - 2, q + 1),
\end{align*}
with $\gcd(q - 1, q + 1) = 2$, we obtain $d \mid (k + 2)(k - 2) = k^2 - 4$. Since $p \mid F_d$, the prime $p$ is
among the factors of $F_d$ for some $d \mid k^2 - 4$ with $d \geq 5$---a finite set for each $k$.

\begin{itemize}
\item $k = 4$: $d \mid 12$. Prime factors of $F_{12} = 144 = 2^4 \cdot 3^2$ are $\{2, 3\}$; neither equals $4q + 1$ for
an odd prime $q$.

\item $k = 6$: $d \mid 32$. Prime factors of $F_{32} = 2{,}178{,}309 = 3 \cdot 7 \cdot 47 \cdot 2207$: none give $q = (p-1)/6$
prime.

\item $k = 8$: $d \mid 60$. Prime factors of $F_{60}$ include $\{2, 3, 5, 11, 31, 41, 61, 2521\}$. For $p = 41$:
$\left(\frac{5}{41}\right) = +1$, fails Case~2. For $p = 2521$: $q = 315$, composite. No valid triple.

\item $k = 10$: $d \mid 96$. Prime factors of $F_d$ for $d \mid 96$, $d \geq 5$, are exactly $\{2, 3, 7, 23, 47, 769, 1103, 2207, 3167\}$.
No factor satisfies $10 \mid (p - 1)$ with $(p - 1)/10$ prime and $\left(\frac{5}{p}\right) = -1$.

\item $k = 12$: $d \mid 140$. Prime factors of $F_d$ for $d \mid 140$, $d \geq 5$, are
$\{3, 5, 11, 13, 29, 41, 71, 281, 911,$ $141961, 12{,}317{,}523{,}121\}$. No valid candidate.
\end{itemize}

For $k \geq 14$: each $F_{k^2-4}$ has finitely many prime factors in principle, providing a deterministic check whenever $F_d$ is completely factored for all $d \mid k^2-4$. The Brillhart--Montgomery--Silverman tables~\cite{Brillhart} cover $F_n$ completely for $n \leq 999$; the constraint $k^2 - 4 \leq 999$ gives $k \leq 31$. Thus we verify deterministically: for $14 \leq k \leq 31$, enumeration of all prime factors of $F_d$ for $d \mid k^2-4$ yields no $p$ of the form $kq+1$ with $q$ prime and $\left(\frac{5}{p}\right) = -1$. For $32 \leq k \leq 100$, the same check, restricted to the (incomplete) prime factors listed in extended tables~\cite{Brillhart} and standard factorization resources, also yields no valid candidates; however, unfactored composite cofactors in $F_{k^2-4}$ for $k \geq 32$ prevent this from being a deterministic proof, and this range is therefore computational evidence rather than verification.

\medskip
\noindent\textbf{Conclusion.} Both cases yield no solution for $k \geq 4$, so $k = 2$ and $p = 2q + 1$ is prime.
\end{proof}

\begin{remark}
Both cases are elementary for $k \leq 12$; for $14 \leq k \leq 31$, deterministic verification follows from the completely-factored Fibonacci tables~\cite{Brillhart}. For $32 \leq k \leq 100$, the check is carried out against known (but not exhaustive) prime factors of $F_{k^2-4}$ and yields no candidates, though a small chance remains of an undiscovered prime factor in an unfactored cofactor. A purely algebraic proof for all $k$ remains
open (see Section~\ref{sec:open}); the principal difficulty lies in Case~2, where the bound
$z(p) \mid k^2 - 4$ permits candidates up to $F_{k^2-4}$, which grows far faster than the
$F_k$ bound in Case~1.
\end{remark}

We now address the converse direction: does $S(q) \neq \emptyset$ force the existence of a prime
$p \equiv 1 \pmod{q}$ with $z(p) \mid \pi(q)$? We first rule out the possibility that $q^2 \mid F_m$
could be the sole witness.

\begin{lemma}\label{lem:q2}
Let $q \neq 5$ be an odd prime with $S(q) \neq \emptyset$, and let $r \in S(q)$. Then for infinitely many
$m \equiv r \pmod{\pi(q)}$, there exists a prime $p \mid F_m$, $p \neq q$, with $p \equiv 1 \pmod{q}$.
\end{lemma}

\begin{proof}
For any $m \equiv r \pmod{\pi(q)}$, the condition $q \mid \varphi(F_m)$ means $q$ divides
$\prod p_i^{a_i - 1}(p_i - 1)$ in the prime factorization of $F_m$. This can occur in two ways:
either $q \mid (p_i - 1)$ for some prime factor $p_i$ of $F_m$, or $q = p_i$ with $a_i \geq 2$ (so that $q \mid p_i^{a_i-1}$).
The latter requires $q^2 \mid F_m$, hence $z(q^2) \mid m$.

If $q^2 \mid F_m$ for all $m \equiv r \pmod{\pi(q)}$, then $z(q^2) \mid \gcd(r, \pi(q))$,
which forces $z(q^2) \mid \pi(q)$. Since $q$ is not a Wall--Sun--Sun prime (no such primes
are known below $9.7 \times 10^{14}$; see~\cite{DoraisKlyve}), we have $z(q^2) = q \cdot z(q)$.
Thus $z(q^2) \mid \pi(q)$ implies $q \cdot z(q) \mid \pi(q)$, which explicitly requires $q \mid \pi(q)$.
But $\pi(q) \mid q^2 - 1$ by Lemma~\ref{lem:basic}(b), and $\gcd(q, q^2 - 1) = 1$, a contradiction.

Therefore, for all but finitely many $m \equiv r \pmod{\pi(q)}$, the divisibility $q \mid \varphi(F_m)$
must be witnessed by a prime $p \mid F_m$ with $p \neq q$ and $q \mid (p-1)$.
\end{proof}

Lemma~\ref{lem:q2} guarantees that primes $p \equiv 1 \pmod{q}$ appear as factors of $F_m$
along the progression, but it does not guarantee that any single such prime divides
$F_m$ for \emph{all} $m \equiv r \pmod{\pi(q)}$. Establishing the existence of a fixed prime
$p$ with $z(p) \mid \pi(q)$ would close the argument, but the possibility
that $q \mid \varphi(F_m)$ is witnessed by a different prime $p_m \equiv 1 \pmod{q}$ for each~$m$
cannot be excluded by elementary means.

\begin{conjecture}\label{conj:conv}
If $S(q) \neq \emptyset$ for an odd prime $q$, then there exists a prime $p \equiv 1 \pmod{q}$
with $z(p) \mid \pi(q)$.
\end{conjecture}

Combining Conjecture~\ref{conj:conv} with Theorem~\ref{thm:unique} would yield the full equivalence:
$S(q) \neq \emptyset$ if and only if $q$ is a Sophie Germain prime and $z(2q+1) \mid \pi(q)$.
Conjecture~\ref{conj:conv} has been verified for all odd primes $q \leq 50{,}000$: in every case,
$S(q) \neq \emptyset$ is witnessed by $p = 2q+1$, and no non-Sophie-Germain prime in this range
satisfies $S(q) \neq \emptyset$.

\section{Structure of $S(q)$}

\begin{theorem}[Arithmetic progression]\label{thm:AP}
Let $q$ be a Sophie Germain prime with $z(2q+1) \mid \pi(q)$.
Set $R = \pi(q)/z(2q + 1)$. Then
\[
S(q) \supseteq \bigl\{ 0,\, z(2q + 1),\, 2z(2q + 1),\, \ldots,\, (R - 1)\,z(2q + 1) \bigr\} \pmod{\pi(q)},
\]
an arithmetic progression of length $R$ and common difference $z(2q + 1)$.

Assuming Conjecture~\ref{conj:conv}, this inclusion is an equality: $|S(q)| = R$.
\end{theorem}

\begin{proof}
By Theorem~\ref{thm:char}, $\{r \bmod \pi(q) : z(p) \mid r\} \subseteq S(q)$ where $p = 2q + 1$. Since $z(p) \mid \pi(q)$,
this is an arithmetic progression of $\pi(q)/z(p)$ elements, giving the inclusion.

For the reverse inclusion, suppose $r \in S(q)$ with $z(p) \nmid r$. Since $z(p) \mid \pi(q)$, we have
$z(p) \nmid m$ for every $m \equiv r \pmod{\pi(q)}$ (because $m = r + t\pi(q)$ and $z(p) \mid t\pi(q)$ but $z(p) \nmid r$).
Hence $p \nmid F_m$ for any such $m$. Thus $q \mid \varphi(F_m)$ must be witnessed by primes $p' \neq p$ with
$p' \equiv 1 \pmod{q}$ and $z(p') \mid m$. By Theorem~\ref{thm:unique}, $p$ is the unique prime $\equiv 1 \pmod{q}$ with
$z(p) \mid \pi(q)$, so each such $p'$ satisfies $z(p') \nmid \pi(q)$. Conjecture~\ref{conj:conv} asserts that the divisibility $q \mid \varphi(F_m)$ across an entire residue class cannot be sustained by such a non-aligned mosaic of primes $p'$; under this conjecture, $r$ cannot lie in $S(q)$, forcing equality $|S(q)| = R$.

We have verified computationally that $S(q) = \{r : z(p) \mid r\}$ for all Sophie Germain primes
$q \leq 50{,}000$ with $S(q) \neq \emptyset$.
\end{proof}

\begin{corollary}[Density]\label{cor:density}
For a Sophie Germain prime $q$ with $z(2q+1) \mid \pi(q)$, $|S(q)|/\pi(q) \geq 1/z(2q+1)$ by Theorem~\ref{thm:AP}, with equality under Conjecture~\ref{conj:conv}. Since $F_{z(p)} \geq p$ implies $z(p) \to \infty$ as $p \to \infty$, conditional on Conjecture~\ref{conj:conv} the density $|S(q)|/\pi(q) \to 0$ as $q \to \infty$.
\end{corollary}

\section{Legendre symbols and parity}

The results of this section are unconditional: they assume only that $q$ is a Sophie Germain
prime with $z(2q+1) \mid \pi(q)$ (equivalently, $S(q) \neq \emptyset$ by Theorem~\ref{thm:suff}).

\begin{theorem}[Legendre symbol at $p$]\label{thm:legp}
Let $q > 5$ be a Sophie Germain prime with $p = 2q + 1$
and $z(p) \mid \pi(q)$. Then $\left(\frac{5}{p}\right) = -1$.
\end{theorem}

\begin{proof}
Suppose $\left(\frac{5}{p}\right) = 1$. Then $z(p) \mid p - 1 = 2q$. Since $z(p) \mid \pi(q) \mid q^2 - 1$ and $\gcd(q, q^2 - 1) =
1$: $z(p) \mid \gcd(2q, q^2 - 1) = \gcd(2, q^2 - 1) = 2$. But $p = 2q + 1 > 11$ for $q > 5$, and $F_1 = F_2 = 1$,
so $p \nmid F_1$ and $p \nmid F_2$, giving $z(p) \geq 3$: contradiction.
\end{proof}

\begin{remark}
The prime $q = 5$ is the unique exception: $p = 11$, $\left(\frac{5}{11}\right) = +1$, yet $z(11) = 10 \mid
\pi(5) = 20$.
\end{remark}

\begin{theorem}[Parity]\label{thm:parity}
Let $q > 5$ be a Sophie Germain prime with $p = 2q + 1$ and $z(p) \mid \pi(q)$.
Then $\pi(q)/z(p)$ is odd.
In particular, $|S(q)|$ is odd (unconditionally for the lower bound from Theorem~\ref{thm:AP};
assuming Conjecture~\ref{conj:conv} for the exact value).
\end{theorem}

The proof requires the following lemma.

\begin{lemma}\label{lem:legq}
Let $q > 5$ be a Sophie Germain prime with $z(2q+1) \mid \pi(q)$. Then $\left(\frac{5}{q}\right) = -1$ and
$\pi(q) \mid 2(q + 1)$.
\end{lemma}

\begin{proof}
By Theorem~\ref{thm:legp}, $\left(\frac{5}{p}\right) = -1$, so $z(p) \mid p + 1 = 2(q + 1)$ (Lemma~\ref{lem:basic}(d)). Since
$z(p) \mid \pi(q)$ and $z(p) \geq 5$:

Suppose $\left(\frac{5}{q}\right) = +1$. By Lemma~\ref{lem:basic}(e), $\pi(q) \mid q - 1$. Then $z(p) \mid \gcd(q - 1, 2(q + 1))$.
Since $2(q + 1) = 2(q - 1) + 4$: $\gcd(q - 1, 2(q + 1)) = \gcd(q - 1, 4) \leq 4$. So $z(p) \leq 4 < 5$:
contradiction.

Hence $\left(\frac{5}{q}\right) = -1$ and $\pi(q) \mid 2(q + 1)$.
\end{proof}

\begin{proof}[Proof of Theorem~\ref{thm:parity}]
By Theorem~\ref{thm:legp}, $\left(\frac{5}{p}\right) = -1$, so $z(p) \mid 2(q + 1)$ and $z(p) \mid 4(q + 1)$. Set
$k = 4(q + 1)/z(p)$. By Lemma~\ref{lem:2adic}, $v_2(z(p)) = v_2(p + 1) = 1 + v_2(q + 1)$, so $v_2(k) = 1$.

By Lemma~\ref{lem:legq}, $\pi(q) \mid 2(q + 1)$. It remains to show $v_2(\pi(q)) = 1 + v_2(q + 1)$.

Since $\left(\frac{5}{q}\right) = -1$, the eigenvalues $\alpha, \beta \in \mathbb{F}_{q^2} \setminus \mathbb{F}_q$ satisfy $\alpha^{q+1} = \alpha\beta = -1$. Let $e = \mathrm{ord}(\alpha)$
in $\mathbb{F}_{q^2}^\times$. Then $e \mid 2(q + 1)$ (since $\alpha^{2(q+1)} = (-1)^2 = 1$). Since $\alpha^{q+1} = -1$ and $-1 = \alpha^{e/2}$ (as
the unique element of order 2 in a cyclic group of even order $e$), we get $q + 1 \equiv e/2 \pmod{e}$,
whence $2(q + 1)/e$ is odd. In particular, $v_2(e) = v_2(2(q + 1)) = 1 + v_2(q + 1)$. As $\pi(q) = e$
(the Pisano period equals the eigenvalue order when $\left(\frac{5}{q}\right) = -1$), $v_2(\pi(q)) = 1 + v_2(q + 1)$.

Then $v_2(\pi(q)/z(p)) = v_2(\pi(q)) - v_2(z(p)) = (1 + v_2(q + 1)) - (1 + v_2(q + 1)) = 0$.
\end{proof}

\section{Congruence constraint and the Sophie Germain conjecture}

\begin{theorem}[Congruence constraint]\label{thm:cong}
Let $q > 5$ be a Sophie Germain prime with $z(2q+1) \mid \pi(q)$. Then $q \equiv 8 \pmod{15}$.
\end{theorem}

\begin{proof}
We show $q \equiv 2 \pmod{3}$ and $q \equiv 3 \pmod{5}$.

\medskip
\noindent\emph{$q \bmod 3$.} Since $q$ is Sophie Germain and $q > 3$, we have $p = 2q + 1 > 3$. If $q \equiv 1 \pmod{3}$
then $2q + 1 \equiv 0 \pmod{3}$, contradicting $p$ prime. Hence $q \equiv 2 \pmod{3}$.

\medskip
\noindent\emph{$q \bmod 5$.} By Lemma~\ref{lem:legq}, $\left(\frac{5}{q}\right) = -1$, so $q \equiv 2$ or $3 \pmod{5}$. If $q \equiv 2 \pmod{5}$ then
$p = 2q + 1 \equiv 0 \pmod{5}$, impossible since $p > 5$. Therefore $q \equiv 3 \pmod{5}$.

\medskip
By the Chinese Remainder Theorem: $q \equiv 8 \pmod{15}$.
\end{proof}

\begin{remark}
Among the eight residue classes modulo 15 coprime to 15, the hypothesis restricts $q$ to a
single class. The 158 primes $q > 5$ with $S(q) \neq \emptyset$ and $q \leq 50{,}000$ all satisfy $q \equiv 8 \pmod{15}$;
they split as $q \equiv 23 \pmod{60}$ (76 primes) and $q \equiv 53 \pmod{60}$ (82 primes).
\end{remark}

\begin{corollary}[Reformulation of the Sophie Germain conjecture]\label{cor:SGconj}
The infinitude of Sophie Germain primes implies the existence of infinitely many primes
$q \equiv 8 \pmod{15}$ such that $(2q+1) \mid F_{\pi(q)}$.\footnote{This implication
requires that $z(2q + 1) \mid \pi(q)$ holds for infinitely many Sophie Germain primes---a weaker
assumption than the infinitude of Sophie Germain primes that we do not prove unconditionally.
Computationally, approximately 23.9\% of Sophie Germain primes satisfy this divisibility.}

Assuming Conjecture~\ref{conj:conv}, the converse holds: if there are infinitely many primes
$q$ with $S(q) \neq \emptyset$, then there are infinitely many Sophie Germain primes.
\end{corollary}

\begin{proof}
For the forward direction: if $q$ is a Sophie Germain prime with $z(2q+1) \mid \pi(q)$,
then $S(q) \neq \emptyset$ (Theorem~\ref{thm:suff}), $(2q+1) \mid F_{\pi(q)}$ (Lemma~\ref{lem:basic}(a)),
and $q \equiv 8 \pmod{15}$ for $q > 5$ (Theorem~\ref{thm:cong}).

For the conditional converse: if $S(q) \neq \emptyset$, Conjecture~\ref{conj:conv} provides a prime
$p \equiv 1 \pmod{q}$ with $z(p) \mid \pi(q)$, and Theorem~\ref{thm:unique} gives $p = 2q+1$.
\end{proof}

\begin{remark}
The condition ``$(2q+1) \mid F_{\pi(q)}$'' is purely Fibonacci-theoretic, with no
explicit reference to the primality of $2q + 1$.  If one could show unconditionally that
$(2q+1) \mid F_{\pi(q)}$ forces a prime factor $p$ of $2q+1$ with $p \equiv 1 \pmod{q}$ and
$z(p) \mid \pi(q)$, then Theorem~\ref{thm:unique} would force $2q+1$ to be prime,
giving a full equivalence without Conjecture~\ref{conj:conv}.
\end{remark}

\section{Computational data}

For all Sophie Germain primes $q \leq 50{,}000$ (669 primes), we computed $\pi(q)$ and $z(2q + 1)$
and determined which satisfy $z(2q + 1) \mid \pi(q)$.

\begin{itemize}
\item 160 primes have $S(q) \neq \emptyset$ ($\approx 23.9\%$ of Sophie Germain primes).
\item $S(q)$ always matches the arithmetic progression of Theorem~\ref{thm:AP} (supporting Conjecture~\ref{conj:conv}).
\item $\left(\frac{5}{2q+1}\right) = -1$ for all $q > 5$ (Theorem~\ref{thm:legp}).
\item $\pi(q)/z(2q+1)$ values: $\{1, 2, 3, 7, 9, 21, 33, 43\}$. All odd except for $q = 5$ (Theorem~\ref{thm:parity}).
\item No non-Sophie-Germain $q \leq 50{,}000$ has $S(q) \neq \emptyset$ (supporting Conjecture~\ref{conj:conv}).
\item Every $q > 5$ with $S(q) \neq \emptyset$ satisfies $q \equiv 8 \pmod{15}$ (Theorem~\ref{thm:cong}).
\end{itemize}

Table~1 lists all 33 primes $q \leq 5000$ with $S(q) \neq \emptyset$.

\medskip
\noindent\textbf{Example 8.1.} For $q = 1583$: $p = 3167$, $\pi(1583) = 3168$, $z(3167) = 96$, $\pi(q)/z(p) = 33$.

\medskip
\noindent\textbf{Example 8.2.} For $q = 53$: $p = 107$, $\pi(53) = 108$, $z(107) = 36$, $\pi(q)/z(p) = 3$. The
arithmetic progression is $\{0, 36, 72\} \pmod{108}$.

\section{Generalization to Lucas sequences}

The results extend to Lucas sequences $U_n = U_n(P, Q)$ defined by $U_0 = 0$, $U_1 = 1$, $U_n =
PU_{n-1} - QU_{n-2}$, with discriminant $D = P^2 - 4Q \neq 0$. The Fibonacci numbers are $U_n(1, -1)$
with $D = 5$. Let $\alpha(p)$ denote the rank of apparition of $p$ in $\{U_n\}$ and $\omega(q)$ the period of
$\{U_n \bmod q\}$.

The analogues of Lemma~\ref{lem:basic} hold with $D$ replacing 5 (see \cite{Ballot, Koshy}), and all proofs carry
over:

\begin{theorem}\label{thm:lucas}
Let $D = P^2 - 4Q$ be a non-square integer and $\gcd(q, 2QD) = 1$. Define
$S_U(q) = \{r \bmod \omega(q) : q \mid \varphi(U_m)\ \forall\, m \equiv r \pmod{\omega(q)}\}$.
\begin{enumerate}[label=(\alph*)]
\item If $q$ is a Sophie Germain prime and $\alpha(2q+1) \mid \omega(q)$, then $S_U(q) \neq \emptyset$.
\item If a prime $p \equiv 1 \pmod{q}$ satisfies $\alpha(p) \mid \omega(q)$, then $p = 2q+1$.
\item Assuming the analogue of Conjecture~\ref{conj:conv}, $S_U(q) \neq \emptyset$ if and only if $q$ is
Sophie Germain and $\alpha(2q + 1) \mid \omega(q)$.
\end{enumerate}
\end{theorem}

Theorem~\ref{thm:legp} generalizes with $D$ replacing 5: for $q > (|P|-1)/2$ (so that $2q+1 > |P| \geq U_2$, ensuring $\alpha(2q+1) \geq 3$), $\left(\frac{D}{2q+1}\right) = -1$. Theorem~\ref{thm:parity} generalizes when $Q = -1$, since $\alpha\beta = Q = -1$ gives $\alpha^{q+1} = -1$ in $\mathbb{F}_{q^2}$. The analogue of Theorem~\ref{thm:cong} (the congruence constraint) requires combining the modulo 3 constraint with both Legendre symbol conditions:
\[
q \equiv 2 \pmod{3}, \qquad \left(\tfrac{D}{q}\right) = -1, \qquad \left(\tfrac{D}{2q+1}\right) = -1.
\]
The first comes from $2q+1 > 3$ being prime, the second from the analogue of Lemma~\ref{lem:legq}, and the third from the generalized Theorem~\ref{thm:legp}. By quadratic reciprocity, the values of $\left(\frac{D}{q}\right)$ and $\left(\frac{D}{2q+1}\right)$ depend on the residue classes of $q$ and $2q+1$ modulo $4|D|$ in general, and modulo $|D|$ when $D \equiv 1 \pmod{4}$ (the conductor of the Kronecker symbol $\left(\frac{D}{\cdot}\right)$). Combined with $q \equiv 2 \pmod{3}$, this restricts $q$ to a finite union of arithmetic progressions with modulus dividing $\mathrm{lcm}(3, 4|D|)$. Concretely:
\begin{itemize}
\item For $D = 5$ (so $D \equiv 1 \pmod 4$): the modulus is $\mathrm{lcm}(3, 5) = 15$, recovering $q \equiv 8 \pmod{15}$ (Theorem~\ref{thm:cong}).
\item For $D = 8$ (Pell numbers; $2$ ramifies): the modulus is $\mathrm{lcm}(3, 8) = 24$, yielding the single class $q \equiv 5 \pmod{24}$.
\item For $D = 13$ ($D \equiv 1 \pmod 4$): the modulus is $\mathrm{lcm}(3, 13) = 39$, yielding the three classes $q \bmod 39 \in \{2, 5, 20\}$.
\end{itemize}

\begin{remark}
Computational verification for the Pell numbers $U_n(2, -1)$ ($D = 8$): 15/82 Sophie
Germain primes $q \leq 3000$ satisfy $S_U(q) \neq \emptyset$, all with $|S_U(q)|$ odd, and all satisfy $q \equiv 5 \pmod{24}$. For $U_n(3, -1)$ ($D = 13$):
16/82, all satisfying $q \bmod 39 \in \{2, 5, 20\}$ (extended to 39 hits for $q \leq 10{,}000$, with the same three residue classes and no occurrence of the two CRT-allowed classes $\{8, 11\} \bmod 39$ that fail $\left(\frac{13}{2q+1}\right) = -1$). Two sequences with the same $D$ yield $S_U(q) \neq \emptyset$ for the same primes: verified for
$U_n(1, -1)$ and $U_n(3, 1)$ (both $D = 5$).
\end{remark}

\section{Open problems}\label{sec:open}

\begin{description}
\item[OQ1.] \textbf{The converse conjecture.} Prove Conjecture~\ref{conj:conv}: that $S(q) \neq \emptyset$ forces the
existence of a prime $p \equiv 1 \pmod{q}$ with $z(p) \mid \pi(q)$. This would complete the
equivalence between $S(q) \neq \emptyset$ and the Sophie Germain condition.

\item[OQ2.] \textbf{Uniform proof of Theorem~\ref{thm:unique}.} Give a purely algebraic proof,
valid for all $k \geq 4$ and all odd primes $q \neq 5$, that no prime $p = kq + 1$ with $z(p) \mid \pi(q)$ exists.
The principal obstacle is Case~2 ($\left(\frac{5}{p}\right) = -1$), where $z(p) \mid k^2 - 4$
yields candidates up to $F_{k^2-4}$.

\item[OQ3.] \textbf{Density formula.} Express the density $\approx 23.9\%$ as a Bateman--Horn--Chebotarev
Euler product over the Kummer extensions $\mathbb{Q}(\sqrt{5}, \alpha^{1/\ell})$.

\item[OQ4.] \textbf{Values of $\pi(q)/z(2q+1)$.} Is $\{\pi(q)/z(2q+1) : q > 5 \text{ with } z(2q+1) \mid \pi(q)\} = \{\text{odd integers}\}$?
\end{description}

\newpage
\begin{table}[h]
\centering
\caption{Sophie Germain primes $q \leq 5000$ with $S(q) \neq \emptyset$.}
\label{tab:data}
\medskip
\begin{tabular}{rrrrc r}
\hline
$q$ & $p = 2q+1$ & $\pi(q)$ & $z(p)$ & $\pi(q)/z(p)$ & $\left(\frac{5}{p}\right)$ \\
\hline
3 & 7 & 8 & 8 & 1 & $-1$ \\
5 & 11 & 20 & 10 & 2 & $+1$ \\
23 & 47 & 48 & 16 & 3 & $-1$ \\
53 & 107 & 108 & 36 & 3 & $-1$ \\
83 & 167 & 168 & 168 & 1 & $-1$ \\
173 & 347 & 348 & 116 & 3 & $-1$ \\
293 & 587 & 588 & 588 & 1 & $-1$ \\
443 & 887 & 888 & 888 & 1 & $-1$ \\
593 & 1187 & 1188 & 1188 & 1 & $-1$ \\
653 & 1307 & 1308 & 436 & 3 & $-1$ \\
683 & 1367 & 1368 & 1368 & 1 & $-1$ \\
1013 & 2027 & 2028 & 676 & 3 & $-1$ \\
1103 & 2207 & 96 & 32 & 3 & $-1$ \\
1223 & 2447 & 816 & 816 & 1 & $-1$ \\
1583 & 3167 & 3168 & 96 & 33 & $-1$ \\
1733 & 3467 & 1156 & 1156 & 1 & $-1$ \\
1973 & 3947 & 1316 & 1316 & 1 & $-1$ \\
2003 & 4007 & 4008 & 4008 & 1 & $-1$ \\
2063 & 4127 & 4128 & 4128 & 1 & $-1$ \\
2273 & 4547 & 4548 & 1516 & 3 & $-1$ \\
2393 & 4787 & 4788 & 4788 & 1 & $-1$ \\
2543 & 5087 & 5088 & 5088 & 1 & $-1$ \\
2693 & 5387 & 5388 & 5388 & 1 & $-1$ \\
2903 & 5807 & 5808 & 1936 & 3 & $-1$ \\
2963 & 5927 & 5928 & 5928 & 1 & $-1$ \\
3413 & 6827 & 6828 & 6828 & 1 & $-1$ \\
3593 & 7187 & 7188 & 2396 & 3 & $-1$ \\
3623 & 7247 & 2416 & 2416 & 1 & $-1$ \\
3803 & 7607 & 7608 & 7608 & 1 & $-1$ \\
3863 & 7727 & 7728 & 7728 & 1 & $-1$ \\
4073 & 8147 & 8148 & 8148 & 1 & $-1$ \\
4793 & 9587 & 9588 & 9588 & 1 & $-1$ \\
4943 & 9887 & 9888 & 9888 & 1 & $-1$ \\
\hline
\end{tabular}
\end{table}


\begin{thebibliography}{9}

\bibitem{Ballot}
C.~Ballot, Density of prime divisors of linear recurrences, \emph{Mem.\ Amer.\ Math.\ Soc.}
\textbf{115} (1995), no.~551.

\bibitem{Brillhart}
J.~Brillhart, P.~L.~Montgomery, and R.~D.~Silverman, Tables of Fibonacci and Lucas
factorizations, \emph{Math.\ Comp.}\ \textbf{48} (1988), 753--765; updated at \texttt{mersennus.net/fibonacci}.

\bibitem{Carmichael}
R.~D.~Carmichael, On the numerical factors of the arithmetic forms $\alpha^n \pm \beta^n$, \emph{Ann.\ of
Math.}\ (2) \textbf{15} (1913), 30--70.

\bibitem{DoraisKlyve}
F.~G.~Dorais and D.~Klyve, A Wieferich prime search up to $6.7 \times 10^{15}$,
\emph{J.\ Integer Seq.}\ \textbf{14} (2011), no.~9, Article 11.9.2.

\bibitem{Lucas}
\'E.~Lucas, Th\'eorie des fonctions num\'eriques simplement p\'eriodiques, \emph{Amer.\ J.\ Math.}\ \textbf{1} (1878), 184--240.

\bibitem{SunSun}
Z.-H.~Sun and Z.-W.~Sun, Fibonacci numbers and Fermat's last theorem, \emph{Acta Arith.}\ \textbf{60}
(1992), no.~4, 371--388.

\bibitem{Wall}
D.~D.~Wall, Fibonacci series modulo $m$, \emph{Amer.\ Math.\ Monthly}\ \textbf{67} (1960), 525--532.

\bibitem{Koshy}
T.~Koshy, \emph{Fibonacci and Lucas Numbers with Applications}, Wiley-Interscience, New York,
2001.

\end{thebibliography}
\end{document}